\documentclass[12pt]{article}%\documentclass[12pt,leqno]{article}

\usepackage{amsfonts,mathrsfs}
\usepackage{amssymb,amsmath,amsbsy,amsthm}
\usepackage{dsfont}
\usepackage{blkarray}
\usepackage{tikz}
\usepackage{tkz-euclide}
\usepackage{tikz-cd}
\usetikzlibrary{patterns}
\usepackage{ae}
\usepackage{bm}
\usepackage{graphicx}
\usepackage{float}
\usepackage{cite}
 \usepackage{indentfirst}
\usepackage{lineno}
\usepackage{titlesec}
\usepackage{enumitem}
\usepackage{color}
\usepackage{aliascnt}
\usepackage{varioref}
\usepackage{hyperref}
\hypersetup{colorlinks=true,linkcolor=cyan,anchorcolor=cyan,citecolor=red,CJKbookmarks=True,
,urlcolor=cyan}
\usepackage{cleveref}

\numberwithin{equation}{section}

\def\int{\mbox{\rm int}}

\def\And{\mbox{\rm ~and~}}

\def\i{\mbox{\rm (\hspace{0.2mm}i\hspace{0.2mm})}\,}

\def\({\mbox{\rm (}}\def\){\mbox{\rm )}}

\makeatletter

\newcommand{\Rmnum}[1]{\expandafter\@slowromancap\romannumeral #1@}
\makeatother
\newtheorem{theorem}{Theorem}[section]
\newaliascnt{lemma}{theorem}
\newtheorem{lemma}[lemma]{Lemma}
\aliascntresetthe{lemma}

\newaliascnt{proposition}{theorem}

\aliascntresetthe{proposition}

\newaliascnt{fact}{theorem}

\aliascntresetthe{fact}

\newaliascnt{definition}{theorem}
\newtheorem{definition}[definition]{Definition}
\aliascntresetthe{definition}

\newaliascnt{conjecture}{theorem}
\newtheorem{conjecture}[conjecture]{Conjecture}
\aliascntresetthe{conjecture}

\newaliascnt{corollary}{theorem}
\newtheorem{corollary}[corollary]{Corollary}
\aliascntresetthe{corollary}

\newaliascnt{claim}{theorem}

\aliascntresetthe{claim}

\newaliascnt{problem}{theorem}

\aliascntresetthe{problem}

\newaliascnt{question}{theorem}

\aliascntresetthe{question}

\newaliascnt{remark}{theorem}

\aliascntresetthe{remark}

\newaliascnt{example}{theorem}
\newtheorem{example}[example]{Example}
\aliascntresetthe{example}

\newaliascnt{notation}{theorem}

\aliascntresetthe{notation}

\begin{document}
\begin{center}
{\Large\bf
Polynomials Counting Group Colorings in Graphs}\\[7pt]
\end{center}
\begin{center}
Houshan Fu\\[5pt]
School of Mathematics and Information Science\\
 Guangzhou University\\
Guangzhou 510006, Guangdong, P. R. China\\[5pt]
Email: fuhoushan@gzhu.edu.cn\\[15pt]
\end{center}
\begin{abstract}
Jaeger et al. in 1992 introduced group coloring as the dual concept to group connectivity in graphs. Let $A$ be an additive Abelian group, $ f: E(G)\to A$ and $D$ an  orientation of a graph $G$. A vertex coloring $c:V(G)\to A$ is an $(A, f)$-coloring if $c(v)-c(u)\ne f(e)$ for each oriented edge $e=uv$ from $u$ to $v$ under $D$. Kochol recently introduced the assigning polynomial to count nowhere-zero chains in graphs--nonhomogeneous analogues of nowhere-zero flows in \cite{Kochol2022}, and later extended the approach to regular matroids in \cite{Kochol2024}.  Motivated by Kochol's work, we define the $\alpha$-compatible graph and  the cycle-assigning polynomial $P(G, \alpha; k)$ at $k$ in terms of $\alpha$-compatible spanning subgraphs, where $\alpha$ is an assigning of $G$ from its cycles to $\{0,1\}$.  We prove that  $P(G,\alpha;k)$ evaluates the number of $(A,f)$-colorings of $G$ for any Abelian group $A$ of order $k$ and $f:E(G)\to A$ such that the assigning $\alpha_{D,f}$ given by $f$ equals $\alpha$.  Such an assigning is admissible. Based on Kochol's work, we derive that $k^{-c(G)}P(G,\alpha;k)$ is a polynomial enumerating $(A,f)$-tensions and counting specific nowhere-zero chains. 

Furthermore, by extending Whitney's broken cycle concept to broken compatible cycles, we show that the absolute value of the coefficient of $k^{|V(G)|-i}$ in $P(G,\alpha;k)$  associated with admissible assignings $\alpha$ equals the number of $\alpha$-compatible spanning subgraphs that have $i$ edges and contain no broken $\alpha$-compatible cycles. According to the combinatorial explanation, we establish a unified order-preserving relation from admissible assignings to cycle-assigning polynomials, and further show that for any admissible assigning $\alpha$ of $G$ with  $\alpha(e)=1$ for every loop $e$, the coefficients of $P(G,\alpha;k)$ are nonzero and alternate in sign.

\vspace{1ex}
\end{abstract}
\section{Introduction}\label{Sec1}
The graphs $G$ with vertex set $V(G)$ and edge set $E(G)$ considered in this paper are finite, allowing multiple edges and loops. A more general setting and background on graphs can be found in the book \cite{Bondy2008}. An {\em orientation} of a graph is an assignment of direction to each edge. In this context, each oriented edge is referred to as an {\em arc}. An {\em oriented graph} arises from a graph $G$ together with a particular orientation $D$, denoted by  $D(G):=\big(V(D),A(D)\big)$, where $V(D)=V(G)$ and $A(D)$ is the collection of arcs under the orientation $D$.  

Graph coloring is a significant subfield in graph theory that originated in the middle of the 19th century with the famous Four Color Problem. A {\em$k$-vertex coloring} of $G$ is a mapping $c:V(G)\to S$, where $S$ is a set of distinct $k$ colors, typically $S=\{1, 2, . . . , k\}$. A vertex coloring $c$ is {\em proper} if  $c(u)\ne c(v)$ whenever vertices $u$ and $v$ are adjacent.

We consider a finite additive Abelian group $A$ with the identity $0$ as coloring sets in this paper. Jaeger, Linial, Payan and Tarsi \cite{Jaeger1992} introduced  the concept of  group coloring as the dual concept of group connectivity of graphs in 1992. Group connectivity and group colorings of graphs are nicely surveyed in \cite{Lai-Shao2011}. Let $ f: E(G)\to A$ be a function. A vertex coloring $c:V(G)\to A$ is called an {\em$(A, f)$-coloring} if $c(v)-c(u)\ne f(e)$ for each oriented edge $e=uv$ from $u$ to $v$ under an orientation $D$ of $G$. When $f\equiv0$, $(A,0)$-coloring is exactly the ordinary proper coloring. Furthermore, $G$ is said to be {\em$A$-colorable} under the orientation $D$ if for any function $ f:E(G)\to A$, $G$ always admits an $(A, f)$-coloring. Since then, topics related to group coloring problems have attracted a lot of research attentions, such as \cite{Kral2005,Langhede-Thomassen2020,Langhede-Thomassen2024,Li-Lai2013,Lai-Li2006,Lai-Mazza2021,Lai-Mazza2021-1,Lai-Zhang2002}.   Their primary focus is on the (weak) group chromatic number and related group coloring problems, such as whether a graph that is $A_1$-colorable for some Abelian group $A_1$ is also $A_2$-colorable for any Abelian group $A_2$ of the same order. However, another fundamental problem of counting $(A, f)$-colorings has been overlooked. In this paper, we introduce the cycle-assigning polynomial, a concept inspired by Kochol's work \cite{Kochol2022,Kochol2024}, which precisely counts the number of $(A, f)$-colorings of graphs and generalizes the classical chromatic polynomial. We primarily investigate its fundamental properties, particularly focusing on its coefficients.

A fundamental problem in graph theory is determining the {\em chromatic number} of a graph, which refers to the smallest integer $k$ such that the graph admits a proper coloring with $k$ colors. Counting proper colorings can be seen as another facet of this problem. In an attempt to address the long-standing Four Color Problem, Birkhoff \cite{Birkhoff1912} discovered that the number of proper colorings of a planar $G$ with $k$ colors is a polynomial in $k$,  known as the {\em chromatic polynomial}. Whitney \cite{Whitney1932,Whitney1932-1} later extended the concept of the chromatic polynomial to general graphs and gave a combinatorial description of its coefficients by introducing the broken cycle, now famously known as the  {\em Whitney's Broken Cycle Theorem}. Since then, various generalizations of this theorem have been developed across different mathematical structures, such as matroids \cite{Brylawski1977,Heron1972}, lattices \cite{Blass-Sagan1997,Rota1964}, hyperplane arrangements \cite{Orlik-Terao1992}, hypergraphs \cite{Dohmen1995}.  See \cite{Dohmen-Trinks2014} for more related work. Motivated by these work, it is natural to ask the two questions:
\begin{itemize}
\item Is there a polynomial function of $|A|$ that generalizes the chromatic polynomial and counts $(A,f)$-colorings in graphs?
\item Can we provide a combinatorial interpretation for the coefficients of the polynomial function?
\end{itemize}

It is worth remarking that the difference between the Abelian group colorings of vertices in an oriented graph gives rise to a group-valued function on arcs, refereed to as a {\em tension},  whose value at each arc is the head-point value minus the tail-point value of the coloring. From this perspective, the proper colorings correspond to nowhere-zero tensions. It is known that the number of nowhere-zero tensions in Abelian groups of orders $k$ is also a polynomial. This was first introduced by Tutte \cite{Tutte1954} and later formally named the {\em tension polynomial} of $G$ by Kochol \cite{Kochol2002}. Note that the tension polynomial is a nontrivial divisor of the chromatic polynomial. For more related work, refer to \cite{Chen2007,Kochol2004}. Likewise, an arbitrary $(A,f)$-coloring of $G$ also produces an $(A,f)$-tension, whose detailed definition will be formally presented in \autoref{Sec3-0}. Naturally, there are two questions regarding $(A,f)$-tensions analogous to the aforementioned ones in $(A,f)$-colorings. An answer to the first question in $(A,f)$-tensions seems to be implicit in the work of Kochol \cite{Kochol2024}. We shall present an explicit answer to this question in \autoref{f-Tension-Polynomial}, and further explore the connection  between $(A,f)$-colorings and nowhere-zero chains related to homomorphisms (see \cite{Kochol2024}) in \autoref{Sec3-0}.

To answer these questions, several significant works should be noted. Most recently, Kochol \cite{Kochol2022} first introduced the assigning polynomial to count nowhere-zero chains in graphs--nonhomogeneous analogues of nowhere-zero flows (called $(A,b)$-flows in \cite{Lai2006}) by the zero-one assigning from certain vertex sets of $G$ to the set $\{0,1\}$, which is a generalization of the classical flow polynomial \cite{Tutte1954}. Subsequently, Kochol \cite{Kochol2024} further extended the approach for represented regular matroids to count nowhere-zero chains associated with homomorphisms and zero-one assignings from circuits of a regular matroid to $\{0,1\}$. Kochol's pioneering work has been a cornerstone of our research. Particularly, the zero-one assignings on cycles of graphs play a central  role in the study of $(A,f)$-colorings and $(A,f)$-tensions in graphs. 

A {\em cycle } $C$ of $G$ is a connected 2-regular subgraph of $G$. The corresponding edge set $E(C)$ is referred to as a {\em circuit} of $G$. Let $\mathcal{C}(G)$ be the family of cycles of $G$. An {\em assigning} of $G$ is a mapping from $\mathcal{C}(G)$ to $\{0,1\}$. Write $\alpha\equiv0$ if $\alpha(C)=0$ for all $C\in\mathcal{C}(G)$ or $\mathcal{C}(G)=\emptyset$. 

To more conveniently investigate the properties of the cycle-assigning polynomials, we use the $\alpha$-compatible graph to mirror the role of both the $(M,\prec)$-compatible and $\delta(M,\alpha;X)$ associated with a matroid $M$ introduced by Kochol \cite{Kochol2024} (see \autoref{Sec3-0} for detailed definitions).
\begin{definition}
{\rm
Let $\alpha$ be an assigning of a graph $G$. $G$ is said to be {\em $\alpha$-compatible} if $\alpha(C)=0$ for its every cycle $C$. A subgraph $H$ of $G$ is called an {\em $\alpha$-compatible subgraph} if the restriction of $\alpha$ to $H$  satisfies $\alpha(C)=0$ for all cycles $C$ in $H$. In this sense, $H$ is $\alpha_H$-compatible, where $\alpha_H$ is the restriction of $\alpha$ to $H$ defined by $\alpha_H(C):=\alpha(C)$ for any cycle $C$ of $H$.}
\end{definition}

We first provide an explicit formula for the number of $(A,f)$-colorings via the $\alpha$-compatible spanning subgraph expansion in \autoref{Counting-Formula}. Inspired by this expression, we define a cycle-assigning polynomial in terms of $\alpha$-compatible spanning subgraphs. As shown in \autoref{Assigning-Polynomial}, this polynomial  precisely count $(A,f)$-colorings of graphs. 
\begin{definition}
{\rm For any assigning $\alpha$ of $G$, the {\em cycle-assigning polynomial} of $G$ is defined as 
\[
P(G,\alpha;k):=\sum_{\substack{\mbox{$H$ is an $\alpha$-compatible}\\ \mbox{spanning subgraph of $G$}}}(-1)^{|E(H)|}k^{c(H)},
\] 
where $c(H)$ is the number of components of $H$. }
\end{definition}
Furthermore, let
\[
\tau(G,\alpha;k):=\sum_{\substack{\mbox{$H$ is an $\alpha$-compatible}\\ \mbox{spanning subgraph of $G$}}}(-1)^{|E(H)|}k^{r(G)-r(H)},
\]
where $r(H):=|V(G)|-c(H)$. To better distinguish both polynomials $P(G,\alpha;k)$ and $\tau(G,\alpha;k)$, we call  $\tau(G,\alpha;k)$ an {\em $\alpha$-assigning polynomial} of $G$ originating from  \cite{Kochol2022,Kochol2024}. It should be noted that the $(A,f)$-tension serves as a bridge connecting $(A,f)$-colorings and nowhere-zero chains associated with regular matroids and homomorphisms (see \autoref{Sec3-0} for detailed discussion).

Due to Kochol's work \cite{Kochol2022,Kochol2024}, we show that $\tau(G,\alpha;k)$ not only enumerates $(A,f)$-tensions of $G$ but also counts the number of certain nowhere-zero chains in $A^{E(G)}$, as stated in \autoref{f-Tension-Polynomial}. As a direct consequence of \cite[Theorem 2]{Kochol2024}, an alternative expression for $P(G,\alpha;k)$ and $\tau(G,\alpha;k)$ is presented in \autoref{Other-Expression}.

To address the second question, we extend the broken cycle \cite{Whitney1932} to the broken $\alpha$-compatible cycle. In \autoref{Generalized-Broken-Circuit-Theorem}, we  show that the unsigned coefficients of cycle-assigning polynomials associated with admissible assignings $\alpha$ can be expressed in terms of $\alpha$-compatible spanning subgraphs that do not contain  broken $\alpha$-compatible cycles. This generalizes the Whitney's celebrated Broken Cycle Theorem \cite{Whitney1932}. Accordingly, we establish a unified order-preserving relation from admissible assignings to cycle-assigning polynomials when both are naturally ordered in \autoref{Comparison}. Additionally, we provide a new perspective on solving group coloring problems in graphs, with detailed discussion in \autoref{Sec4}.

Notice that the $\alpha$-assigning polynomial $\tau(G,\alpha;k)$ is a nontrivial divisor of the cycle-assigning polynomial $P(G,\alpha;k)$, that is, $P(G,\alpha;k)=k^{c(G)}\tau(G,\alpha;k)$.
Thus, the coefficients of $\tau(G,\alpha;k)$ share the same properties as those of  $P(G,\alpha;k)$. 

The paper is organized as follows. \autoref{Sec2} is devoted to studying the counting function for the number of $(A,f)$-colorings and the cycle-assigning polynomial. In \autoref{Sec3-0}, we mainly explain the close connection between cycle-assigning polynomials and Kochol's work \cite{Kochol2024}.  In \autoref{Sec3} and \autoref{Sec4}, we focus on exploring the properties of coefficients of cycle-assigning polynomials associated with admissible assignings.

\section{Cycle-assigning polynomials}\label{Sec2}
In this section, we focus on the study of cycle-assigning polynomials, exploring their structural properties and their roles in counting $(A,f)$-colorings. Let us review some necessary notations and definitions on graphs. Throughout this paper, we always denote by $S^E$ the collection of mappings from a finite set $E$ to a set $S$, unless otherwise stated. The elements of $S^E$ are considered as vectors indexed by $E$. An edge with identical ends is called a {\em loop}, and an edge with distinct ends is a {\em link}. A graph $H$ is a {\em subgraph} of $G$ if $V(H)\subseteq V(G)$ and $E(H)\subseteq E(G)$. When $V(H)=V(G)$, the subgraph $H$ of $G$ is referred to as a {\em spanning subgraph}. For an edge subset $S\subseteq E(G)$, $G-S$ is the spanning subgraph obtained by removing the edges in $S$. Denote by $G|S$ the spanning subgraph with edges $S$, that is, $G|S=G-(E(G)\smallsetminus S)$. Conversely, if $H$ is a proper subgraph which does not contain the edge $e \in E(G)$ with  both ends of $e$ are in $V(H)$, then we denote  by $H+e$ the subgraph of $G$ by adding the edge $e$ to $H$.  In addition, if $e$ is a link of $G$, the graph $G/e$ is obtained by contracting $e$, that is, identifying the ends of $e$ to create a single new vertex and then removing $e$.  Note that even the graph $G$ is simple, the contraction $G/e$ may not be simple, and repeated contractions may also result in loops.

Unless otherwise stated, we always choose an arbitrary but fixed reference orientation $D$ for $G$. Notice that every cycle $C$ of $G$ has exactly two possible directions. We arbitrarily choose and fix a direction $o(C)$ for every cycle $C$. Associated with each cycle $C$ of $G$, the orientation $D$ of $G$ induces a map $\eta_{D(C)}$ from the edge set $E(G)$ to the set $\{0,\pm1\}$ defined by
\begin{equation}\label{Primitive-Chain}
\eta_{D(C)}(e):=
\begin{cases}
1,&  \mbox{if  $e\in E(C)$, the direction of $e$ is consistent with $o(C)$};\\
-1,& \mbox{if  $e\in E(C)$, the direction of $e$ is opposite to $o(C)$};\\
0, & \mbox{if  $e\in E(G)\smallsetminus E(C)$}.
\end{cases}
\end{equation}
Remarkably, each function $f:E(G)\to A$ automatically induces an assigning $\alpha_{D,f}$ on $\mathcal{C}(G)$ such that for each cycle $C$ of $G$,
\[
\alpha_{D,f}(C):=\begin{cases}
0,& \mbox{ if } \sum_{e\in E(C)}\eta_{D(C)}(e)f(e)=0;\\
1,& \mbox{ otherwise}.
\end{cases}
\]
We say that $\alpha=\alpha_{D,f}$ is $D$-{\em admissible}, or simply {\em admissible}. 

Although, the map $\eta_{D(C)}$ induced by a cycle $C$ depends on the arbitrarily chosen direction of $C$, different directions of $C$ only cause a possible $-1$ factor to the $\eta_{D(C)}$ value on each edge of $C$. Thus,  whether $\sum_{e\in E(C)}\eta_{D(C)}(e)f(e)$ equals $0$ is independent of the choice of direction of $C$. Consequently, the choice of direction for cycles has no influence on the assigning $\alpha_{D,f}$, and hence does not affect our results.

To derive an explicit formula for the number of $(A,f)$-colorings, the M\"obius Inversion Formula in \cite[Theorem 2.25]{Bondy2008} is needed. It is stated that if the function $g: 2^{E(G)}\to A$ is defined as
\[
g(S):=\sum_{S\subseteq X\subseteq E(G)}h(X) \quad\mbox{ for }\quad \forall\,S\subseteq E(G),
\]
then, for all $S\subseteq E(G)$,
\begin{equation}\label{Inversion}
h(S)=\sum_{S\subseteq X\subseteq E(G)}(-1)^{|X|-|S|}g(X),
\end{equation}
where $h$ is a function from the family $2^{E(G)}$ of all edge subsets of $G$ to $A$. 

Let $f:E(G)\to A$, and $P\big(D(G),f;A\big)$ denote the number of $(A,f)$-colorings of $G$ under $D$. The result below shows that $P\big(D(G),f;A\big)$ can be expanded in terms of  $\alpha_{D,f}$-compatible spanning subgraphs. 
\begin{theorem}[Counting Formula]\label{Counting-Formula}
Let $ f: E(G)\to A$ and $\alpha=\alpha_{D,f}$. Then
\[
P\big(D(G),f;A\big)=\sum_{\substack{\mbox{$H$ is an $\alpha$-compatible}\\ \mbox{spanning subgraph of $G$}}} (-1)^{|E(H)|}|A|^{c(H)}.
\]
\end{theorem}
\begin{proof}
For any $S\subseteq E(G)$, let $C_S\big(D(G),f;A\big)$ be the number of  vertex colorings $c:V(G)\to A$ such that $c(v)-c(u)=f(e)$ for each oriented edge $e=uv\in S$ from $u$ to $v$, and $P_S\big(D(G),f;A\big)$ be the number of vertex colorings $c:V(G)\to A$ such that $c(v)-c(u)=f(e)$ for each oriented edge $e=uv\in S$ and $c(v)-c(u)\ne f(e)$ for each oriented edge $e=uv\in E(G)\smallsetminus S$. We have
\[
C_S\big(D(G)|S,f;A\big)=C_S\big(D(G),f;A\big),\quad\quad P\big(D(G),f;A\big)=P_{\emptyset}\big(D(G),f;A\big)
\]
and 
\[
C_{\emptyset}\big(D(G),f;A\big)=\sum_{S\subseteq E(G)}P_S\big(D(G),f;A\big).
\]
It follows from \eqref{Inversion} that
\begin{equation}\label{Inclusion-Exclusion}
P\big(D(G),f;A\big)=\sum_{S\subseteq E(G)}(-1)^{|S|}C_S\big(D(G),f;A\big)=\sum_{S\subseteq E(G)}(-1)^{|S|}C_S\big(D(G)|S,f;A\big).
\end{equation}
In \cite[Proposition 2.2]{Kochol2002}, it is demonstrated that $G$ admits a vertex coloring $c$ such that $c(v)-c(u)=f(e)$ for any oriented edge $e=uv$ of $G$ if and only if $f$ satisfies the condition $\sum_{e\in E(C)}\eta_{D(C)}(e)f(e)=0$ for every cycle $C$ of $G$. This is further equivalent to $\alpha(C)=0$ for all cycles $C$ of $G$ since $\alpha=\alpha_{D,f}$. Thus, $C_S\big(D(G)|S,f;A\big)=0$ whenever $G|S$ is an $\alpha$-incompatible subgraph of $G$. Then, $P\big(D(G),f;A\big)$ in \eqref{Inclusion-Exclusion} is simplified to the form
\begin{equation}\label{Eq1}
P\big(D(G),f;A\big)=\sum_{\substack{S\subseteq E(G),\\G|S\mbox{ is $\alpha_{G|S}$-compatible}}}(-1)^{|S|}C_S\big(D(G)|S,f;A\big).
\end{equation}

Denote by  $C_S\big(D(G)|S,f_S;A\big)$ the collection of  vertex colorings $c:V(G)\to A$ of $G|S$ such that $c(v)-c(u)=f(e)$ for each oriented edge $e=uv\in S$. Then, fixing a vertex coloring $c\in C_S\big(D(G)|S,f_S;A\big)$, every vertex coloring $c'\in C_S\big(D(G)|S,f_S;A\big)$ gives rise to a unique vertex coloring $c'-c\in C_S\big(D(G)|S,0_S;A\big)$. Conversely, every vertex coloring $c''\in C_S\big(D(G)|S,0_S;A\big)$ also gives rise to a unique vertex coloring $c+c''\in C_S\big(D(G)|S,f_S;A\big)$. Notice
\[
C_S\big(D(G)|S,f;A\big)=|C_S\big(D(G)|S,f_S;A\big)|\quad\And\quad C_S\big(D(G)|S,0;A\big)=|C_S\big(D(G)|S,0_S;A\big)|.
\]
It follows that
\begin{equation*}
C_S\big(D(G)|S,f;A\big)=C_S\big(D(G)|S,0;A\big)=|A|^{c(G|S)}.
\end{equation*}
Together with \eqref{Eq1}, we have
\[
P\big(D(G),f;A\big)=\sum_{\substack{S\subseteq E(G),\\G|S\mbox{ is $\alpha_{G|S}$-compatible}}}(-1)^{|S|}|A|^{c(G|S)},
\]
which completes the proof.
\end{proof}
As a direct consequence of \autoref{Counting-Formula}, the cycle-assigning polynomials corresponding to admissible assignings precisely count the number of $(A,f)$-colorings of  $G$. 
\begin{theorem}\label{Assigning-Polynomial}
Let $\alpha$ be an admissible assigning of $G$. Then $P(G,\alpha;k)=P\big(D(G),f;A\big)$ for every orientation $D$ of $G$, every Abelian group $A$ of order $k$ and every function $f:E(G)\to A$ satisfying $\alpha_{D,f}=\alpha$.
\end{theorem}

By applying \autoref{Assigning-Polynomial}, we immediately obtain the following corollary. This implies that the number of $(G,f)$-colorings depends only on the assigning $\alpha_{D,f}$, and not on the structure of $A$ and the choice of the orientation $D$.
\begin{corollary}\label{A=B}
Let  $D,D'$ be distinct orientations of $G$, and $A,A'$ be two Abelian groups of the same order and functions $f:E(G)\to A$, $f':E(G)\to A'$. If $\alpha_{D,f}=\alpha_{D',f'}$, then $P\big(D(G),f;A\big)=P\big(D'(G),f';A'\big)$.
\end{corollary}

We further explore the fundamental properties of cycle-assigning polynomials. Suppose $e\in E(G)$ is a link of $G$. Then the set $\mathcal{C}(G/e)$ of cycles of $G/e$ consists of two parts: 
\[
\mathcal{C}_{-e}(G):=\big\{C\in\mathcal{C}(G/e)\cap\mathcal{C}(G):e\notin E(C)\big\},\;\mathcal{C}_{/e}(G/e):=\big\{C/e:C\in\mathcal{C}(G),e\in E(C)\big\}.
\]
In this context, each assigning $\alpha$ of $G$ induces an assigning $\alpha_{G/e}$ of $G/e$, where $\alpha_{G/e}(C):=\alpha(C)$ for $C\in \mathcal{C}_{-e}(G)$ and $\alpha_{G/e}(C/e):=\alpha(C)$ for $C\in \mathcal{C}(G)$ containing $e$. If $e$ is a loop of $G$, then $\mathcal{C}(G-e)=\mathcal{C}(G/e)\subseteq\mathcal{C}(G)$. In this case, each assigning $\alpha$ of $G$ naturally induces assignings $\alpha_{G-e}$ of $G-e$ and  $\alpha_{G/e}$ of $G/e$, where $\alpha_{G-e}(G)(C)=\alpha_{G/e}(C):=\alpha(C)$ for each $C\in\mathcal{C}(G-e)$.

Next, we introduce a fundamental property concerning admissible assignings.
\begin{lemma}\label{Admissible}
Let $\alpha$ be a $D$-admissible assigning of $G$, $T$ be a spanning forest of $G$ and $f_T:E(T)\to A$. Then there exists a function $f:E(G)\to A$ such that $\alpha_{D,f}=\alpha$ and $f(e)=f_T(e)$ for any $e\in E(T)$.
\end{lemma}
\begin{proof}
Since $\alpha$ is a $D$-admissible assigning of $G$, there exists a function $g:E(G)\to A$ such that $\alpha_{D,g}=\alpha$. For each edge $e\in E(G)\smallsetminus E(T)$, let $C_e$ be the unique cycle (fundamental cycle) in the spanning subgraph $T+e$. Define a new function $f:E(G)\to A$ as follows:
\[
f(e)=\begin{cases}
f_T(e),& \text{ if } e\in E(T);\\
\sum_{h\in E(C_e)}\eta_{D(C_e)}(h)g(h)-\sum_{h\in E(C_e)\smallsetminus\{e\}}\eta_{D(C_e)}(h)f_T(h),& \mbox{ if } e\in E(G)\smallsetminus E(T).
\end{cases}
\]
Consequently, for any $e\in E(G)\smallsetminus E(T)$, we have
\[
\sum_{h\in E(C_e)}\eta_{D(C_e)}(h)f(h)=\sum_{h\in E(C_e)}\eta_{D(C_e)}(h)g(h).
\]
In \cite[Corollary 4.11]{Bondy2008}, it is demonstrated that every cycle $C$ of $G$ can be expressed uniquely as a symmetric difference of fundamental cycles  $C_e$ for all $e\in E(C)\smallsetminus E(T)$.
This implies that
\[
\eta_{D(C)}=\sum_{e\in E(C)\smallsetminus E(T)}\delta_{C,C_e}\eta_{D(C_e)},
\]
where $\delta_{C,C_e}=1$ if the direction of $e$ is either consistent with or opposite to both $o(C)$ and $o(C_e)$, and $\delta_{C,C_e}=-1$ otherwise. Therefore, for every cycle $C$ of $G$, we have
\begin{align*}
\sum_{h\in E(C)}\eta_{D(C)}(h)f(h)&=\sum_{h\in E(C)}\sum_{e\in E(C)\smallsetminus E(T)}\delta_{C,C_e}\eta_{D(C_e)}(h)f(h)\\
&=\sum_{e\in E(C)\smallsetminus E(T)}\delta_{C,C_e}\sum_{h\in E(C)}\eta_{D(C_e)}(h)f(h)\\
&=\sum_{e\in E(C)\smallsetminus E(T)}\delta_{C,C_e}\sum_{h\in E(C_e)}\eta_{D(C_e)}(h)g(h)\\
&=\sum_{h\in E(C)}\eta_{D(C)}(h)g(h).
\end{align*}
It follows that $\alpha_{D,f}=\alpha_{D,g}=\alpha$. This completes the proof.
\end{proof}

We now present a deletion-contraction formula for cycle-assigning polynomials, which provides an inductive approach to proving the Generalized Whitney's Broken Cycle Theorem (see \autoref{Generalized-Broken-Circuit-Theorem}).
\begin{theorem}[Deletion-contraction Formula]\label{DCF}
Let $\alpha$ be an admissible assigning of $G$. For any $e\in E(G)$,  $\alpha_{G-e}$ and $\alpha_{G/e}$ are admissible assignings of $G-e$ and $G/e$, respectively, and 
\[
P(G,\alpha;k)=
\begin{cases}
k^{|V(G)|},&\mbox{ if } E(G)=\emptyset;\\
\frac{k-1}{k}P(G-e, \alpha_{G-e};k),& \mbox{ if } e \mbox{ is a bridge};\\
\alpha(e)P(G-e, \alpha_{G-e};k),& \mbox{ if } e \mbox{ is a loop};\\
P(G-e, \alpha_{G-e};k)-P(G/e, \alpha_{G/e};k),& \mbox{ otherwise}.
\end{cases}
\]
\end{theorem}
\begin{proof}
Clearly, $\alpha_{G-e}$ is admissible for any $e\in E(G)$, and $\alpha_{G/e}$ is admissible if $e$ is a loop. When $e$ is a link, from \autoref{Admissible}, we can choose a function $f: E(G)\to A$ such that $\alpha_{D,f}=\alpha$ and $f(e)=0$. Then $\alpha_{G/e}=\alpha_{D,f_{-e}}$, where $f_{-e}$ is the restriction of $f$ to $E(G)\smallsetminus\{e\}$. Consequently, $\alpha_{G/e}$ is admissible .

Let $\mathcal{S}_{e}$ be the family of $\alpha$-compatible spanning subgraphs of $G$  containing $e$ and $\mathcal{S}_{-e}$ be the family of $\alpha$-compatible spanning subgraphs of $G$ that do not contain $e$. If $e$ is a bridge, then a subgraph $H$ containing $e$ is an $\alpha$-compatible spanning subgraph of $G$ if and only if $H-e$ is an $\alpha$-compatible spanning subgraph of $G$. Therefore, we conclude that every $\alpha$-compatible spanning subgraph $H$ of $G$ containing $e$ gives rise to a unique $\alpha$-compatible spanning subgraph $H-e$ of $G$ that does not contain $e$, and vice versa. Then we have
\[
P(G,\alpha;k)=\sum_{H\in\mathcal{S}_{-e}}(-1)^{|E(H+e)|}k^{c(H+e)}+\sum_{H\in\mathcal{S}_{-e}}(-1)^{|E(H)|}k^{c(H)}.
\]
Since $e$ is a bridge, $c(H+e)=c(H)-1$ for any $H\in\mathcal{S}_{-e}$. Hence, we deduce
\[
P(G,\alpha;k)=\frac{k-1}{k}P(G-e, \alpha_{G-e};k).
\]

If $e$ is a loop, we need to consider both cases: $\alpha(e)=0$ and $\alpha(e)=1$. For the former case, similarly to the case where $e$ is a bridge, we can easily derive $P(G,\alpha;k)=0$. For the latter case, it is clear that $H$ is an $\alpha$-compatible spanning subgraph of $G$ if and only if $H$ is an $\alpha_{G-e}$-compatible spanning subgraph of $G-e$. So, we arrive at $P(G,\alpha;k)=P(G-e, \alpha_{G-e};k)$.

If $e$ is neither a bridge nor a loop, then $H$ is an $\alpha$-compatible spanning subgraph of $G$ containing no $e$ if and only if $H$ is an $\alpha_{G-e}$-compatible spanning subgraph of $G-e$, and $H$ is an $\alpha$-compatible spanning subgraph of $G$ containing $e$ if and only if $H/e$ is an $\alpha_{G/e}$-compatible spanning subgraph of $G/e$. Therefore, we have
\begin{align*}
P(G,\alpha;k)&=\sum_{H\in\mathcal{S}_{-e}}(-1)^{|E(H)|}k^{c(H)}+\sum_{H\in\mathcal{S}_{e}}(-1)^{|E(H)|}k^{c(H)}\\
&=P(G-e, \alpha_{G-e};k)+\sum_{H\in\mathcal{S}_{e}}(-1)^{|E(H/e)|+1}k^{c(H/e)}\\
&=P(G-e, \alpha_{G-e};k)-P(G/e, \alpha_{G/e};k),
\end{align*}
which completes the proof.
\end{proof}

In addition, the general cycle-assigning polynomial $P(G,\alpha;k)$ has the following simple decomposition formula.
\begin{theorem}\label{Decomposition}
Let $\alpha$ be an assigning of $G$ and $G=G_1\sqcup G_2\sqcup\cdots\sqcup G_i$, where $G_j$ is the component of $G$ for $j=1,2,\ldots,i$. Then
\[
P(G,\alpha;k)=\prod_{j=1}^iP(G_j,\alpha_{G_j};k).
\]
\end{theorem}
\begin{proof}
Since $G$ decomposes into the union of the components $G_j$ of $G$, each $\alpha$-compatible spanning subgraph of $G$ is partitioned into the disjoint union of the $\alpha_{G_j}$-compatible spanning subgraph of $G_j$, and vice versa. Therefore, we have $P(G,\alpha;k)=\prod_{j=1}^iP(G_j,\alpha_{G_j};k)$.
\end{proof}
At the end of the section, we list the cycle-assigning polynomials for the empty graph, a tree and the cycle on $n$ vertices.
\begin{example}
{\rm Note that the following assignings are admissible obviously.
\begin{itemize}
\item [{\rm(1)}] Let $O_n$ be the empty graph on $n$ vertices. Then $\alpha\equiv0$ and
\[
P(O_n,0;k)=k^n.
\]
\item [{\rm(2)}] Let $T_n$ be a tree on $n$ vertices. Then $\alpha\equiv0$ and
\[
P(T_n,0;k)=k(k-1)^{n-1}=\sum_{i=1}^n(-1)^{n-i}\binom{n-1}{i-1}k^i.
\]
\item [{\rm(3)}] Let $C_n$ be a cycle on $n$ vertices. Combining part (2) and \autoref{DCF}, we have
\[
P(C_n,\alpha;k)=
\begin{cases}
(k-1)^n+(-1)^n(k-1),&\mbox{ if } \alpha(C_n)=0;\\
(k-1)^n+(-1)^{n+1},&\mbox{ if } \alpha(C_n)=1.
\end{cases}
\]
\end{itemize}
}
\end{example}
\section{Polynomials counting nowhere-zero chains associated with homomorphisms}\label{Sec3-0}
In this section, we mainly clarify the close relationship between cycle-assigning polynomials and $\alpha$-assigning polynomials, which was first introduced by Kochol \cite{Kochol2022} to count nowhere-zero chains in graphs--nonhomogeneous analogues of nowhere-zero flows. Most recently,  Kochol \cite{Kochol2024} further extended the approach to regular matroids.  For this purpose, we first review some basic definitions about regular chain groups associated to regular matroids. For additional properties of regular chain groups, refer to the references  \cite{Arrowsmith-Jaeger1982,Tutte1956}.

A regular matroid $M=M(U)$ on a ground set $E$ is represented by a totally unimodular matrix $U$. A more general setting and background on matroids can be found in Oxley's book \cite{Oxley2011}. A {\em chain} on $E$ over an Abelian group $A$ is a map $f\in A^E$, and its support $\sigma(f)$ is $\sigma(f):=\{f(e)\ne 0\mid e\in E\}$. We say that $f$ is {\em proper} if $\sigma(f)=E$.  Associated with $U$ or $M(U)$, a {\em regular chain group} $N$ on $E$ over $\mathbb{Z}$ is the set of integral vectors orthogonal to each row of $U$. The set of chains orthogonal to every chain of $N$ forms a chain group called {\em orthogonal to} $N$, denoted by $N^\perp$. Actually, $N^\perp$ consists of  the integral vectors from the linear combinations of the rows of $U$. A chain $f$ of $N$ is {\em elementary} if there is a no nonzero chain $g$ of $N$ such that $\sigma(g)\subset \sigma(f)$. An elementary chain $f$ of $N$ is said to be {\em primitive} if  $f(e)\in\{\pm1,0\}$ for all $e\in E$. In this section, we always assume that a regular chain group $N$ is associated with a totally unimodular matrix $U(N)$ representing a regular matroid $M=M(N)$.

Let $A$ be an additive Abelian group. We shall consider $A$ as a (right) $\mathbb{Z}$-module such that the scalar multiplication $a\cdot z$ of $a\in A$ by $z\in\mathbb{Z}$ equals zero if $z=0$, $\sum_1^za$ if $z>0$, and $\sum_1^{-z}(-a)$ if $z<0$. Analogously, if $a\in A$ and $f\in \mathbb{Z}^E$, then define $a\cdot f\in A^E$ such that $(a\cdot f)(e)=a\cdot f(e)$ for every $e\in E$. We shall denote by $A(N)$ the $A$-chain group generated by the regular chain group $N$, that is,
\[
A(N):=\Big\{\sum_{i=1}^ma_i\cdot f_i\mid a_i\in A,f_i\in N,m\ge 1\Big\}.
\]
Let $A[N]$ be the collection of all proper $A$-chains from $A(N)$, i.e.,
\[
A[N]:=\big\{f\in A(N)\mid \sigma(f)=E\big\}.
\]
From \cite[Proposition 1]{Arrowsmith-Jaeger1982}, we have
\begin{equation}\label{Regular-Chain-Element}
g\in A^E\text{ is from } A(N)\Longleftrightarrow \sum_{e\in E}g(e)\cdot f(e)=0 \text{ for any } f\in N^\perp.
\end{equation}

Let $\psi$ be a homomorphism from a regular chain group $N$ to $A$.  Define by 
\[
A_\psi[N]:=\Big\{g\in (A-0)^E\mid \sum_{e\in E}g(e)\cdot f(e)=\psi(f) \text{ for all } f\in N\Big\}.
\]
Each element from $A_\psi[N]$ is called a {\em proper $(N,\psi)$-chain}.

Let $\mathcal{P}(N)$ be the set of primitive chains of a regular chain group $N$, and $\mathcal{C}(M)$ be the family of circuits of $M=M(N)$. By \cite{Tutte1956}, $\mathcal{P}(N)$ and $\mathcal{C}(M)$ have the following relationship
\begin{equation}\label{Primitive-Circuit}
\mathcal{C}(M)=\big\{\sigma(c)\mid c\in\mathcal{P}(N)\big\}.
\end{equation}

An {\em assigning} of $M$ is any mapping $\alpha$ from $\mathcal{C}(M)$ to $\{0,1\}$. Write $\alpha\equiv0$ if $\alpha(C)=0$ for all $C\in\mathcal{C}(M)$ or $\mathcal{C}(M)=\emptyset$. Let $e\in E$. It follows from the relation $\mathcal{C}(M-e)\subseteq\mathcal{C}(M)$ that an assigning $\alpha$ of $M$ naturally gives rise to an assigning $\alpha[-e]$ of $M-e$ for which
\[
\alpha[-e](C)=\alpha(C),\quad C\in\mathcal{C}(M-e).
\]
By \cite[Proposition 3.1.10]{Oxley2011}, we know that
\[
\text{if $C\in\mathcal{C}(M/e)$, then either $C\in\mathcal{C}(M)$ or $C\cup e\in\mathcal{C}(M)$.}
\]
Similarly, the assigning $\alpha$ of $M$ also induces an assigning $\alpha[/e]$ of $M/e$ such that $\alpha[/e](C)=\alpha(C)$ if $C\in\mathcal{C}(M/e)\cap\mathcal{C}(M)$, and $\alpha[/e](C)=\alpha(C\cup e)$ otherwise. 

From \eqref{Primitive-Circuit}, the homomorphism $\psi$ automatically leads to an assigning $\alpha_{N,\psi}$ of $M$ such that for every $c\in\mathcal{P}(N)$,
\[  
\alpha_{N,\psi}\big(\sigma(c)\big):=
\begin{cases}
0,& \text{if } \psi(c)=0;\\
1,& \text{if } \psi(c)\ne0.
\end{cases}
\]
An assigning $\alpha$ of $M$ is  said to be {\em homogeneous} if  there exists a homomorphism $\psi:N\to A$ such that $\alpha=\alpha_{N,\psi}$. Furthermore, $\psi$ and $\alpha=\alpha_{N,\psi}$  are {\em proper} if $A_\psi[N]\ne\emptyset$.

In 2024, Kochol showed in \cite[Theorem 1]{Kochol2022} that there exists a polynomial $p(M,\alpha;k)$ counting the number of nowhere-zero chains in $A_\psi[N]$. The polynomial $p(M,\alpha;k)$ is called an {\em $\alpha$-assigning polynomial} of $M$ by Kochol.
\begin{theorem}[\cite{Kochol2024}, Theorem 1]\label{Kochol-Result1} Suppose $M$ is a regular matroid on $E$ and $\alpha$ is a homogeneous assigning of $M$. Then, there exists a polynomial $p(M,\alpha;k)$, such that $p(M,\alpha;k)=|A_{\psi}[N]|$ for every regular chain group $N$ on $E$ and $M(N)=M$, every Abelian group $A$ of order $k$, and every homomorphism $\psi: N\to A$ satisfying $\alpha_{N,\psi}=\alpha$. If $\alpha$ is proper, then $p(M,\alpha;k)$ has degree $r(M)$, and $p(M,\alpha;k)=0$  otherwise. Furthermore, for any $e\in E$, $\alpha[-e]$ and $\alpha[/e]$ are homogeneous assignings of $M-e$ and  $M/e$, respectively, and
\begin{equation*}
p(M,\alpha;k)=
\begin{cases}
1,&\text{if $E=\emptyset$};\\
(k-1)p\big(M-e,\alpha[-e];k\big),&\text{if $e$ is a bridge};\\
\alpha(e)p\big(M-e,\alpha[-e];k\big),&\text{if $e$ is a loop};\\
p\big(M-e,\alpha[-e];k\big)-p\big(M/e,\alpha[/e];k\big),&\text{otherwise}.\\
\end{cases}
\end{equation*}
\end{theorem}

Below we shall explore some applications of the $\alpha$-assigning polynomial to colorings and tensions of $G$. The {\em incidence matrix} of $D(G)$ is the $|V(D)|\times |A(D)|$ matrix $M_{D(G)}:=(m_{va})$ whose rows and columns are indexed by the vertices and arcs, where, for a vertex $v$ and arc $a$
\[
m_{va}:=\begin{cases}
1, & \text{ if arc } a \text{ is a link and vertex } v \text{ is the head of } a;\\
-1,& \text{ if arc } a \text{ is a link and vertex } v \text{ is the tail of } a;\\
0,& \text{ otherwise}.
\end{cases}
\]
If $c\in A^{V(G)}$, then the {\em coboundary} of $c$ is the map $\delta c\in A^{E(G)}$ given by 
\[
\delta c(e):=\sum_{v\in V(G)}m_{vD(e)}c(v),
\]
where $D(e)$ denotes the corresponding arc of the edge $e$ under $D$. In this sense, we call a mapping $t\in A^{E(G)}$ a {\em tension} or an {\em$A$-tension} if there exists a vertex coloring  $c\in  A^{V(G)}$ such that $t=\delta c$. Similar to $(A,f)$-coloring associated with a mapping $f\in A^{E(G)}$, a tension $t$ is referred to as an {\em$(A,f)$-tension} if it satisfies $t(e)\ne f(e)$ for all $e\in E(G)$. In particular, taking $f\equiv0$, the $(A,0)$-tension is exactly the {\em nowhere-zero $A$-tension}. It is clear that all $A$-tensions form an (additive) Abelian subgroup of $A^{E(G)}$, denoted by $\Delta(G,A)$.   Let $\Delta_f(G,A)$ be the set of all $(A,f)$-tensions of $G$ for $f\in A^{E(G)}$. The number of nowhere-zero tensions is evaluated by a polynomial $\tau(G;k)$ in $k$ introduced by Tutte \cite{Tutte1954}. This was developed and  formally named the {\em tension polynomial} of $G$ by Kochol \cite{Kochol2002} in  2002. Motivated by their work, a natural question arises:
\begin{itemize}
\item Is there a polynomial function of $|A|$ that generalizes the tension polynomial and counts $(A,f)$-tensions in graphs?
\end{itemize}
The answer to the question was implicit in the work of Kochol \cite{Kochol2024}. By applying \autoref{Kochol-Result1},  we shall present an explicit answer in \autoref{f-Tension-Polynomial}.

Let $M(G)$ be the cycle matroid of $G$ represented by the incidence matrix $M_{D(G)}$ ($M(G)$ is a regular matroid, see \cite[Proposition 5.1.5]{Oxley2011}), and $\Gamma(G)$ be the regular chain group on $E(G)$ associated with $M_{D(G)}$. From \cite{Tutte1956},  the mapping $\eta_{D(C)}$ defined in \eqref{Primitive-Chain} is the primitive chain of  $\Gamma(G)$ for every cycle $C$ of $G$, and $\Gamma(G)$ is generated by $\big\{\eta_{D(C)}\mid C\in\mathcal{C}(G)\big\}$. In addition, the regular chain group $\Delta(G)$ on $E(G)$ over $\mathbb{Z}$ consisting of all linear combinations of rows from the matrix $M_{D(G)}$ is exactly orthogonal
to  $\Gamma(G)$, i.e., $\Delta(G)=\Gamma(G)^\perp$. Again, from \cite{Tutte1956}, the chain group $\Delta(G,A)$ can be generated by $\Delta(G)$. Then we have the following result.
\begin{corollary}\label{f-Tension-Polynomial}
Let $\alpha$ be an admissible assigning of $G$. Then $\tau(G,\alpha;k)$ and $p\big(M(G),\alpha_{\Gamma(G),\psi};k\big)$ are the same polynomial, and they satisfy that
\[
\tau(G,\alpha;k)=p\big(M(G),\alpha_{\Gamma(G),\psi};k\big)=|A_\psi[\Gamma(G)]|=|\Delta_f(G,A)|
\]
for every Abelian group $A$ of order $k$, every function $f:E(G)\to A$ satisfying $\alpha_{D,f}=\alpha$, and every homomorphism $\psi:\Gamma(G)\to A$ such that  $\psi(\eta_{D(C)})=-\sum_{e\in E(G)}f(e)\cdot\eta_{D(C)}(e) $ for any cycle $C$ of $G$. Furthermore, $\tau(G,\alpha;k)$ has degree $|V(G)|-c(G)$ if $\alpha(e)=1$ for any loop $e$ of $G$ and equals zero otherwise, and $\tau(G,\alpha;k)$ also satisfies
\begin{equation*}
\tau(G,\alpha;k)=
\begin{cases}
1,&\text{if $E(G)=\emptyset$};\\
(k-1)\tau(G-e,\alpha_{G-e};k),&\text{if $e$ is a bridge};\\
\alpha(e)\tau(G-e,\alpha_{G-e};k),&\text{if $e$ is a loop};\\
\tau(G-e,\alpha_{G-e};k)-\tau(G/e,\alpha_{G/e};k),&\text{otherwise}.\\
\end{cases}
\end{equation*}
\end{corollary}

To prove \autoref{f-Tension-Polynomial}, the following lemmas are required.

\begin{lemma}\label{Bijection}
Given a mapping $f\in A^{E(G)}$. Let 
\[
A_f[G]=\big\{g\in (A-0)^{E(G)}\mid g+f\in \Delta(G,A)\big\}.
\]
Then the mapping $\phi:\Delta_f(G,A)\to A_f[G]$ sending $g\in\Delta_f(G,A)$ to $\phi(g)=g-f$ is a bijection.
\end{lemma}
\begin{proof}
Taking any element $g\in\Delta_f(G,A)$, by $g(e)\ne f(e)$ for all $e\in E(G)$, then $\phi(g)\in (A-0)^{E(G)}$ and $\phi(g)+f=g\in \Delta_f(G,A)\subseteq \Delta(G,A)$. Namely, $\phi$ is well-defined. Clearly, $\psi$ is injective.

Conversely, define the mapping $\phi^{-1}:A_f[G]\to \Delta_f(G,A)$ sending $g\in A_f[G]$ to $\phi^{-1}(g)=g+f$. Given an element $g\in A_f[G]$. Since $g(e)\ne 0$ for every $e\in E(G)$, $\phi^{-1}(g)\in \Delta_f(G,A)$. Thus, $\phi^{-1}$ is well-defined. The mapping $\phi^{-1}$ is an injection obviously. Therefore, $\phi$ is bijective. 
\end{proof}

\begin{lemma}\label{A1=B1} Suppose  $f\in A^{E(G)}$ and the homomorphism $\psi:\Gamma(G)\to A$ such that $\psi(\eta_{D(C)})=-\sum_{e\in E(G)}f(e)\cdot\eta_{D(C)}(e)$ for any cycle $C$ of $G$. Then $A_\psi[\Gamma(G)]=A_f[G]$.
\end{lemma}
\begin{proof}
We first show $A_\psi[\Gamma(G)]\subseteq A_f[G]$. Note that the chain group $\Delta(G,A)$ is generated by the regular chain group $\Delta(G)$ and $\Delta(G)$ is orthogonal to $\Gamma(G)$. Then, from \eqref{Regular-Chain-Element}, this proof reduces to verifying that for any $g\in A_\psi[\Gamma(G)]$, $\sum_{e\in E(G)}(g+f)(e)\cdot h(e)=0$ for all $h\in \Gamma(G)$. Since $\Gamma(G)$ is generated by $\big\{\eta_{D(C)}\mid C\in\mathcal{C}(G)\big\}$, the proof  is further equivalent to checking that for any $g\in A_\psi[\Gamma(G)]$, 
\[
\sum_{e\in E(G)}(g+f)(e)\cdot \eta_{D(C)}(e)=0,\quad\forall\, C\in\mathcal{C}(G).
\]
Since $\sum_{e\in E(G)}g(e)\cdot \eta_{D(C)}(e)=\psi(\eta_{D(C)})$ for every $g\in A_\psi[\Gamma(G)]$, it follows from $\psi(\eta_{D(C)})=-\sum_{e\in E(G)}f(e)\cdot\eta_{D(C)}(e) $ that $\sum_{e\in E(G)}(g+f)(e)\cdot \eta_{D(C)}(e)$ is indeed equal to zero.  This implies that $g+f\in \Delta(G,A)$ for any $g\in A_\psi[\Gamma(G)]$, that is, $A_\psi[\Gamma(G)]\subseteq A_f[G]$.

For the opposite inclusion, for any $g\in A_f[G]$, we have $g+f\in\Delta(G,A)$. Again, using \eqref{Regular-Chain-Element}, we have $\sum_{e\in E(G)}(g+f)(e)\cdot \eta_{D(C)}(e)=0$ for all cycles $C$ of $G$. It follows that 
\[
\sum_{e\in E(G)}g(e)\cdot \eta_{D(C)}(e)=-\sum_{e\in E(G)}f(e)\cdot \eta_{D(C)}(e)=\psi( \eta_{D(C)}).
\]
This means $g\in A_\psi[\Gamma(G)]$ for each $g\in A_f[G]$. Namely, we have $A_\psi[\Gamma(G)]\supseteq A_f[G]$. Then we arrive at $A_\psi[\Gamma(G)]=A_f[G]$.
\end{proof}
Now we proceed to prove \autoref{f-Tension-Polynomial}.
\begin{proof}[Proof of  \autoref{f-Tension-Polynomial}]
Note that any $(A,f)$-coloring $c$ of $G$ yields an $(A,f)$-tension $\delta c$ of $G$,  and vice versa. We assert that for every $(A,f)$-tension $t$, there are exactly $|A|^{c(G)}$ $(A,f)$-colorings $c$ satisfying $\delta c=t$. Suppose $t$ is an $(A,f)$-tension of $G$ satisfying $t=\delta c$ for some $c\in A^{V(G)}$. We construct a mapping $c'\in A^{V(G)}$ such that $c'(u)-c(u)=c'(v)-c(v)$ whenever vertices $u$ and $v$ belong to the same component of $G$. It follows that $c'$ is a $(A,f)$-coloring and $\delta c'=t$. On the other hand, if $\delta c^{''}=t$ for another $(A,f)$-coloring $c^{''}$, then $c^{''}(u)-c(u)=c^{''}(v)-c(v)$ whenever $u$ and $v$ are in the same component of $G$. Thus, we have shown the assertion. It follows from \autoref{Assigning-Polynomial} that 
\begin{equation}\label{A}
\tau(G,\alpha;k)=|\Delta_f(G,A)|.
\end{equation}
Applying  \autoref{Kochol-Result1}, we have 
\begin{equation}\label{B}
p\big(M(G),\alpha_{\Gamma(G),\psi};k\big)=|A_\psi[\Gamma(G)]|.
\end{equation}
By \autoref{Bijection} and \autoref{A1=B1}, we obtain
\begin{equation}\label{C}
|\Delta_f(G,A)|=|A_\psi[\Gamma(G)]|.
\end{equation}
Immediately, by integration of \eqref{A}, \eqref{B} and \eqref{C}, we arrive at 
\[
\tau(G,\alpha;k)=p\big(M(G),\alpha_{\Gamma(G),\psi};k\big)=|A_\psi[\Gamma(G)]|=|\Delta_f(G,A)|.
\]
Note that both $\tau(G,\alpha;k)$ and $p\big(M(G),\alpha_{\Gamma(G),\psi};k\big)$ are polynomials in $k$ of degree at most $|V(G)|-c(G)$, and they agree on infinitely many positive integers $k$. Hence, from the fundamental theorem of algebra, these two polynomials are the same.

In fact, we have $\alpha_{\Gamma(G),\psi}=\alpha_{D,f}=\alpha$ from \eqref{Primitive-Chain}. In \cite[Theorem 3]{Kochol2024}, it is stated that
\[
\text{a homogeneous assigning $\alpha$ of $M(G)$ is proper $\Longleftrightarrow$ $\alpha(e)=1$ for each loop $e$ of $M(G)$}.
\]
Together with \autoref{Kochol-Result1} and $r\big(M(G)\big)=|V(G)|-c(G)$, we have that $\tau(G,\alpha;k)$ has degree $|V(G)|-c(G)$ if $\alpha(e)=1$ for any loop $e$ of $G$ and equals zero otherwise. Finally, the remaining part is a direct result of \autoref{DCF} via $\tau(G,\alpha;k)=k^{-c(G)}P(G,\alpha;k)$.
\end{proof}

In \cite[Proposition 2.2]{Kochol2002}, it is demonstrated that $f\in A^{E(G)}$ is a tension of $G$ if and only if $f$ satisfies the condition $\sum_{e\in E(C)}\eta_{D(C)}(e)f(e)=0$ for every cycle $C$ of $G$. It follows that $\alpha_{D,f}\equiv 0$ for every $f\in\Delta(G,A)$. So, the assigning polynomial $\tau(G,\alpha_{D,f};k)$ is the classical tension polynomial of $G$ (see \cite{Kochol2002}) whenever $f\in\Delta(G,A)$.

Let $M$ be a matroid on the ground  set $E$ and $\prec$ be a linear ordering on $E$. For every subset $X\subseteq E$, the notation ${\rm min}(X)$ represents the minimal element of $X$ with respect to $\prec$. A subset  $X\subseteq E$ is {\em $(M,\prec)$-compatible} if $C\cap X\ne\{{\rm min}(C)\}$ for each $C\in\mathcal{C}(M)$, where $\mathcal{C}(M)$ is the set of all circuits of $M$. Furthermore, we use $\mathcal{E}(M,\prec)$ to denote the collection of all $(M^*,\prec)$-compatible subsets of $E$, where $M^*$ is the dual matroid of $M$. For any $X\subseteq E$ and assigning $\alpha$ of $M$, we define $\delta(M,\alpha;X)$ to be
\[
\delta(M,\alpha;X)=:
\begin{cases}
0,&\text{if there exists $C\in\mathcal{C}(M)$ such that $C\subseteq X$ and $\alpha(C)=1$};\\
1,&\text{otherwise}.
\end{cases}
\]  

For a regular matroid $M$, Kochol derived  an explicit expression for the $\alpha$-assigning polynomial $p(M,\alpha;k)$ over all $(M^*,\prec)$-compatible subsets of $E$ in \cite[Theorem 2]{Kochol2024}.
\begin{theorem}[\cite{Kochol2024}, Theorem 2]\label{Linear-Ordering}
Let $\alpha$ be a homogeneous assigning of a regular matroid $M$ on $E$ and $\prec$ be a linear ordering of $E$. Then,
\[
p(M,\alpha;k)=\sum_{X\in\mathcal{E}(M,\prec)}\delta(M,\alpha;X)(-1)^{|X|}(k-1)^{r(M/X)}.
\]
\end{theorem}

A {\em bond} of a graph is a minimal nonempty edge cut, which is the smallest nonempty subset of edges whose removal from the graph results in an increase in the number of components. Below we return to the cycle matroid $M(G)$ on $E(G)$ represented by $M_{D(G)}$. Its dual matroid $M^*(G)$, also on $E(G)$, is referred to as a {\em bond matroid} of $G$ ($G$ may be not a planar graph). In \cite[Proposition 2.3.1]{Oxley2011},  it is demonstrated that
\[
\text{$X\subseteq E(G)$ is a circuit of $M^*(G)$ if and only if $X$ is a bond of $G$}.
\] 
Let $\prec$ be a linear ordering on $E(G)$. Define $\mathcal{E}(G,\prec)$ as 
\[
\mathcal{E}(G,\prec):=\big\{X\subseteq E(G)\mid X\cap C^*\ne \{{\rm min}(C^*)\}\text{ for all bonds $C^*$ of  $G$}\big\}.
\]
It follows that
\begin{equation}\label{A2=B2}
\mathcal{E}(G,\prec)=\mathcal{E}\big(M(G),\prec\big).
\end{equation}

For any $X\subseteq E(G)$ and assigning $\alpha$ of $G$, we define $\delta(G,\alpha;X)$  by
\[
\delta(G,\alpha;X)=:
\begin{cases}
0,&\text{if there exists $C\in\mathcal{C}(G)$ such that $E(C)\subseteq X$ and $\alpha(C)=1$};\\
1,&\text{otherwise}.
\end{cases}
\]  
From \eqref{Primitive-Chain}, every admissible assigning $\alpha=\alpha_{D,f}$ of $G$ induced by $f\in A^{E(G)}$ is precisely a homogeneous assigning $\alpha_{\Gamma(G),\psi}$ of $M(G)$, where the homomorphism $\psi:\Gamma(G)\to A$ satisfies $\psi(\eta_{D(C)})=-\sum_{e\in E(G)}f(e)\cdot\eta_{D(C)}(e) $ for any cycle $C$ of $G$. In this sense, we have
\begin{equation}\label{A3=B3}
\delta(G,\alpha;X)=\delta\big(M(G),\alpha;X\big),\quad\text{ for every }X\subseteq E(G).
\end{equation}

By integrating the preceding discussions, as a direct consequence of \autoref{f-Tension-Polynomial} and \autoref{Linear-Ordering}, we conclude an analogous result of \autoref{Linear-Ordering} concerning $\tau(G,\alpha;k)$ and $P(G,\alpha;k)$ in the following corollary. 

\begin{corollary}\label{Other-Expression}
Let $\alpha$ be an admissible assigning of  $G$ and $\prec$ a linear ordering on $E(G)$. Then
\[
\tau(G,\alpha;k)=\sum_{X\in \mathcal{E}(G,\prec)}\delta(G,\alpha;X)(-1)^{|X|}(k-1)^{r(G)-r(X)}
\]
and
\[
P(G,\alpha;k)=\sum_{X\in \mathcal{E}(G,\prec)}\delta(G,\alpha;X)(-1)^{|X|}(k-1)^{c(G|X)},
\]
where $r$ is the rank function of $G$ given by $r(X):=|V(G)|-c(G|X)$ for $X\subseteq E(G)$.
\end{corollary}
\begin{proof}
Suppose $\alpha$ is equal to the admissible assigning $\alpha_{D,f}$ for some $f\in A^{E(G)}$. Together with the relations in \eqref{A2=B2} and \eqref{A3=B3}, by using \autoref{f-Tension-Polynomial} and \autoref{Linear-Ordering}, we have
\[
\tau(G,\alpha;k)=\sum_{X\in \mathcal{E}(G,\prec)}\delta(G,\alpha;X)(-1)^{|X|}(k-1)^{r(M(G)/X)}.
\]
Note that  $r\big(M(G)/X\big)=r\big(M(G)\big)-r_{M(G)}(X)$ from \cite[Proposition 3.1.6]{Oxley2011}, where $r_{M(G)}$ is the rank function of $M(G)$. It follows from
$r_{M(G)}(X)=r(X)$ for any $X\subseteq E(G)$ that 
\[
\tau(G,\alpha;k)=\sum_{X\in \mathcal{E}(G,\prec)}\delta(G,\alpha;X)(-1)^{|X|}(k-1)^{r(G)-r(X)}.
\]
Combining the relations $r(X)=|V(G)|-c(G|X)$ and $P(G,\alpha;k)=k^{c(G)}\tau(G,\alpha;k)$, we derive the second formula in \autoref{Other-Expression}. 
\end{proof}
\section{Generalization of Whitney's Broken Cycle Theorem}\label{Sec3}
A natural question is whether there is a combinatorial interpretation for the coefficients of cycle-assigning polynomials. To address this, we introduce the concept of a broken $\alpha$-compatible cycle as an extension of the usual broken cycle proposed by Whitney \cite{Whitney1932}. 
\begin{definition}\label{Broken-Circuit}
{\rm Let $\alpha$ be an assigning of $G$ and $\prec$ be a linear ordering on $E(G)$. A subgraph $H$ of $G$ is said to be:
\begin{itemize}
\item an {\em $\alpha$-compatible cycle} of $G$ if $H$ is a cycle and $\alpha(H)=0$;
\item a {\em  broken $\alpha$-compatible cycle} of $G$ if $H$ is obtained from an $\alpha$-compatible cycle by removing the maximal element with respect to the linear ordering $\prec$.
\end{itemize}
}
\end{definition}
The following result provides a counting interpretation for the coefficients of the cycle-assigning polynomial $P(G,\alpha;k)$ via broken $\alpha$-compatible cycle, extending the Whitney's celebrated Broken Cycle Theorem  \cite{Whitney1932}. It is worth noting that the coefficients $w_i(G,\alpha)$ in \autoref{Generalized-Broken-Circuit-Theorem} are independent of the choice of edge order. 
\begin{theorem}[Generalized Whitney's Broken Cycle Theorem]\label{Generalized-Broken-Circuit-Theorem}
Let $\alpha$ be an admissible assigning of $G$ and $\prec$ a linear ordering on $E(G)$. Write $P(G,\alpha;k)=\sum_{i=0}^{r(G)}(-1)^{i}w_i(G,\alpha)k^{|V(G)|-i}$. Then each unsigned coefficient $w_i(G,\alpha)$ is the number of $\alpha$-compatible spanning subgraphs of $G$ that have exactly $i$ edges and do not contain  broken $\alpha$-compatible cycles with respect to $\prec$.
\end{theorem}
\begin{proof}
Let $E(G)=\{e_m\prec\cdots\prec e_2\prec e_1\}$. We prove the result by induction on $m$. If $m=0$,  there is nothing to prove. Suppose $m\ge 1$ and the result holds for smaller values of $m$. Let $e_m=uv$. If $e_m$ is a loop and $\alpha(e_m)=0$, it is trivial via \autoref{DCF}. If $e_m$ is a loop and $\alpha(e_m)=1$, from 
\[
P(G,\alpha;k)=P(G-e, \alpha_{G-e};k)
\]
in \autoref{DCF}, the result holds in this case by the inductive hypothesis. 

If $e_m$ is a link, by the inductive hypothesis, the result holds for both graphs $G-e_m$ and $G/e_m$. With a slight abuse of notation, we identify $E(G/e_m)=\{e_{m-1}\prec\cdots\prec e_2\prec e_1\}$. More specifically, by identifying the ends of $e_m$, this create a single new vertex of $G/e_m$, denoted by $v_{e_m}$. Suppose the edge $e_i=xy$ in $G$ for some $i<m$. If neither $x$ nor $y$ is one of the vertices $u$ or $v$, then $e_i$  automatically remains an edge in $G/e_m$.  If either $x$ or $y$ is one of the vertices $u$ or $v$, we identity $e_i$ as $uv_{e_m}$, $v_{e_m}v$ or $v_{e_m}v_{e_m}$ depending on the incidence of $e_i$ in $G$. By the Deletion-contraction Formula in \autoref{DCF}, we have $w_0(G,\alpha)=w_0(G-e_m,\alpha_{G-e_m})$ and
\begin{equation}\label{Eq4}
w_i(G,\alpha)=w_i(G-e_m,\alpha_{G-e_m})+w_{i-1}(G/e_m,\alpha_{G/e_m}),\; i=1,2,\ldots,r(G).
\end{equation}
Therefore, to complete the induction step, it is sufficient to check that the relation in \eqref{Eq4} is true. But this is a consequence of the following two simple facts: (a) $H$ is an $\alpha$-compatible spanning subgraph $H$ of $G$ that does not contain $e$ and broken $\alpha$-compatible cycles  if and only if $H$ is an $\alpha_{G-e_m}$-compatible spanning subgraph of $G-e_m$ that does not contain   broken $\alpha_{G-e_m}$-compatible  cycles; (b) $H$ is an $\alpha$-compatible spanning subgraph $H$ of $G$ that contains $e$ and no  broken $\alpha$-compatible cycles  if and only if $H/e_m$ is an $\alpha_{G/e_m}$-compatible spanning subgraph of $G/e_m$ that does not contain   broken $\alpha_{G/e_m}$-compatible cycles. 
\end{proof}

In the remainder of this section, we present a direct proof of  \autoref{Generalized-Broken-Circuit-Theorem}. For this purpose, the next lemma is required.
\begin{lemma}\label{No-Broken-Circuit}
Let $\alpha$ be a $D$-admissible assigning of $G$ and $\prec$ a linear ordering on $E(G)$. Suppose $H$ is a subgraph of $G$ containing a broken $\alpha$-compatible cycle $B$ and $B=C-e_C$ for some $\alpha$-compatible cycle $C$ of $G$, where $e_C$ is the maximal member in $E(C)$ with respect to $\prec$. If $H$ is an $\alpha$-compatible subgraph of $G$ and $e_C\notin E(H)$, then $H+e_C$ is an $\alpha$-compatible subgraph.
\end{lemma}
\begin{proof}
Note that $C$ is a cycle in $H+e_C$, but not in $H$. If $H+e_C$ contains exactly one more cycle $C$ than $H$, there is nothing left to prove. Otherwise, $H+e_C$ contains at least one additional cycle compared to $H$, excluding the cycle $C$. Every such additional cycle, together with the cycle $C$, forms a theta-subgraph in $H+e_C$. Therefore, this case reduces to verifying that each such the theta-subgraph in $H+e_C$ is an $\alpha$-compatible subgraph of $G$. Without loss of generality, we only need to examine the following oriented theta-subgraph in $H+e_C$ with the orientation $D$, as shown in \autoref{Fg1}.
\begin{figure}[H]
\centering
\begin{tikzpicture}[scale=0.7,line width=1pt]
\def\radius{3}
 \foreach \i in {1,...,8} {
        \pgfmathsetmacro{\angleA}{360/8 * (\i - 1)}
        \pgfmathsetmacro{\angleB}{360/8 * \i}
        \coordinate (P\i) at ({\radius * cos(\angleA)}, {\radius * sin(\angleA)});
        
        \fill (P\i) circle(2pt);
        \node at ({1.2*\radius * cos(360/8 * 0}, {1.2*\radius * sin(360/8 * 0)}) {$u_i$}; 
        \node at ({1.2*\radius * cos(360/8 * 1}, {1.2*\radius * sin(360/8 * 1)}) {$u_{i+1}$}; 
        \node at ({1.2*\radius * cos(360/8 * 3}, {1.2*\radius * sin(360/8 * 3)}) {$u_j$}; 
        \node at ({1.2*\radius * cos(360/8 * 4}, {1.2*\radius * sin(360/8 * 4)}) {$v_k$};  
        \node at ({1.2*\radius * cos(360/8 * 5}, {1.2*\radius * sin(360/8 * 5)}) {$v_{k+1}$};  
        \node at ({1.2*\radius * cos(360/8 * 7}, {1.2*\radius * sin(360/8 * 7)}) {$v_l$};  
    }
   \draw[-stealth,blue] (P1) arc[start angle=0,end angle=45,radius=\radius];     
    \draw[-stealth,dashed,blue] (P2) arc[start angle=45,end angle=90,radius=\radius];  
    \draw[-stealth,dashed,blue] (P3) arc[start angle=90,end angle=135,radius=\radius]; 
    \draw[-stealth,blue] (P4) arc[start angle=135,end angle=180,radius=\radius];  
    \draw[-stealth,green] (P5) arc[start angle=180,end angle=225,radius=\radius];  
    \draw[-stealth,dashed,green] (P6) arc[start angle=225,end angle=270,radius=\radius]; 
    \draw[-stealth,dashed,green] (P7) arc[start angle=270,end angle=315,radius=\radius]; 
    \draw[-stealth,green] (P8) arc[start angle=315,end angle=360,radius=\radius];  
\foreach \i in {1,...,7} {
        \pgfmathsetmacro{\fraction}{\i/7} 
        \coordinate (S\i) at ({(1-\fraction)*\radius*cos(0) + \fraction*\radius*cos(180)}, 
                              {(1-\fraction)*\radius*sin(0) + \fraction*\radius*sin(180)});
        \fill (S\i) circle(2pt);  
        \node at ({(1-1/7)*\radius*cos(0) + 1/7*\radius*cos(180)}, 
                  {(1-1/7)*\radius*sin(0) + 1/7*\radius*sin(180) - 0.3}) {$u_{i-1}$};  
        \node at ({(1-2/7)*\radius*cos(0) + 2/7*\radius*cos(180)}, 
                  {(1-2/5)*\radius*sin(0) + 2/5*\radius*sin(180) - 0.3}) {$u_1$}; 
        \node[red] at ({(1-3/7)*\radius*cos(0) + 3/7*\radius*cos(180)}, 
                  {(1-3/7)*\radius*sin(0) + 3/7*\radius*sin(180) - 0.3}) {$u_0$};  
        \node[red] at ({(1-4/7)*\radius*cos(0) + 4/7*\radius*cos(180)}, 
                  {(1-4/7)*\radius*sin(0) + 4/7*\radius*sin(180) - 0.3}) {$v_0$}; 
        \node at ({(1-5/7)*\radius*cos(0) + 5/7*\radius*cos(180)}, 
                  {(1-5/7)*\radius*sin(0) + 5/7*\radius*sin(180) - 0.3}) {$v_1$}; 
        \node at ({(1-6/7)*\radius*cos(0) + 6/7*\radius*cos(180)}, 
                  {(1-6/7)*\radius*sin(0) + 6/7*\radius*sin(180) - 0.3}) {$v_{k-1}$}; 
  }
\draw[-stealth,red] (P1) -- (S1);  
\draw[-stealth,dashed,red] (S1) -- (S2);
\draw[-stealth,red] (S2) -- (S3);
\draw[-stealth,red] (S3) -- (S4);
\draw[-stealth,red] (S4) -- (S5); 
\draw[-stealth,dashed,red] (S5) -- (S6); 
\draw[-stealth,red] (S6) -- (S7); 
\node[blue] at (0, \radius -0.5) {$P_3$}; 
\node[red] at (0, \radius -2.5) {$P_1$};    
\node[green] at (0, \radius -5.5) {$P_2$}; 
\end{tikzpicture}
\caption{The directed paths $P_1,P_2$ and $P_3$ toghter form an oriented theta-graph}\label{Fg1}
\end{figure}
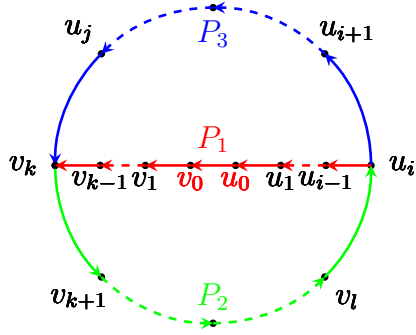

More specifically, in \autoref{Fg1},  the edge $e_C=u_0v_0$; the cycle $C$, which is the union of the red directed path $P_1=(u_i,u_{i-1},\ldots,v_{k-1},v_k)$ and the green directed path $P_2=(v_k,v_{k+1},\ldots,v_l,u_i)$, and the cycle $C_1$, which is the union of  the red directed path $P_1$ and the blue directed path $P_3=(u_i,u_{i+1},\ldots,u_j,v_k)$, are in $H+e_C$ but not in $H$; the cycle $C_2$ is the common cycle of $H+e_C$ and $H$, formed by the union of both directed paths $P_2$ and $P_3$. Then, to complete our proof, it suffices to show $\alpha(C)=\alpha(C_1)=\alpha(C_2)=0$.

Since $\alpha$ is $D$-admissible, we may assume $\alpha=\alpha_{D,f}$ for some Abelian group $A$ and function $f:E(G)\to A$. Combining that $C$ and $H$ are $\alpha$-compatible subgraphs of $G$, this means 
\[
\alpha_{D,f}(C)=\alpha(C)=0\quad\And\quad\alpha_{D,f}(C_2)=\alpha(C_2)=0.
\]
This is further equivalent to satisfying the following equations
\[
\sum_{e\in E(C)}\eta_{D(C)}(e)f(e)=\sum_{e\in E(P_1)}\eta_{D(C)}(e)f(e)+\sum_{e\in E(P_2)}\eta_{D(C)}(e)f(e)=0
\]
and
\[
\sum_{e\in E(C_2)}\eta_{D(C_2)}(e)f(e)=\sum_{e\in E(P_2)}\eta_{D(C_2)}(e)f(e)+\sum_{e\in E(P_3)}\eta_{D(C_2)}(e)f(e)=0.
\]
Taking the clockwise as the reference direction for every cycle, then $\eta_{D(C_1)}(e)=-\eta_{D(C)}(e)$ for any edge $e\in E(P_1)$ and $\eta_{D(C_1)}(e)=\eta_{D(C_2)}(e)$ for any edge $e\in E(P_3)$. Thus, we deduce
\[
\sum_{e\in E(C_1)}\eta_{D(C_1)}(e)f(e)=\sum_{e\in E(P_1)}\eta_{D(C_1)}(e)f(e)+\sum_{e\in E(P_3)}\eta_{D(C_1)}(e)f(e)=0.
\]
This implies $\alpha(C_1)=\alpha_{D,f}(C_1)=0$. Therefore, $H+e_C$ is an $\alpha$-compatible subgraph of $G$.
\end{proof}

We now return to the proof of the Generalized Whitney's Broken Cycle Theorem.
\begin{proof}[Proof of \autoref{Generalized-Broken-Circuit-Theorem}]
Suppose $G$ has $q\ge 0$  broken $\alpha$-compatible cycles. For every  broken $\alpha$-compatible cycle $B$, let $e_B$ be the maximal member in $B$ with respect to the linear ordering $\prec$. We can arrange the  broken $\alpha$-compatible cycles $B_1,\ldots, B_q$ such that $e_{B_1}\preceq e_{B_2}\cdots\preceq  e_{B_q}$. Let $\mathcal{S}_1$ be the collection of $\alpha$-compatible spanning subgraphs of $G$ containing the  broken $\alpha$-compatible cycle $B_1$. For $i=2,\ldots,q$, let $\mathcal{S}_i$ be the collection of $\alpha$-compatible spanning subgraphs of $G$  containing the  broken $\alpha$-compatible cycle $B_i$ but not any of $B_1,\ldots, B_{i-1}$. In addition, let $\mathcal{S}_{q+1}$ be the collection of $\alpha$-compatible spanning subgraphs of $G$ containing no  broken $\alpha$-compatible cycles. Then $\{\mathcal{S}_1,\mathcal{S}_2,\ldots,\mathcal{S}_{q+1}\}$ decomposes the set of $\alpha$-compatible spanning subgraphs of $G$. Combining \autoref{Assigning-Polynomial}, the cycle-assigning polynomial $P(G,\alpha;k)$ can be written in the form
\begin{equation}\label{Eq2}
P(G,\alpha;k)=\sum_{i=1}^{q+1}\sum_{H\in\mathcal{S}_i}(-1)^{|E(H)|}k^{c(H)}.
\end{equation}

We first consider $\mathcal{S}_1$. Assume $B_1=C_1-e_{C_1}$ for some $\alpha$-compatible cycle $C_1$ of $G$, where $e_{C_1}$ is the maximal member in $C_1$ with respect to $\prec$. Clearly each member $H$ in $\mathcal{S}_1$ containing $e_{C_1}$  gives rise to a unique member $H-e_{C_1}$ in $\mathcal{S}_1$ that does not contain $e_{C_1}$. Conversely, according to \autoref{No-Broken-Circuit}, each member $H$ in $\mathcal{S}_1$ containing no $e_{C_1}$ gives rise to a unique member $H+e_{C_1}$ in $\mathcal{S}_1$ containing $e_{C_1}$. In addition, since $B_1+e_{C_1}=C_1$ is a cycle of $G$, for each $H\in\mathcal{S}_1$ containing $e_{C_1}$, we have $c(H)=c(H- e_{C_1})$. Thus, we derive
\[
\sum_{H\in\mathcal{S}_1,\, e_{C_1}\in E(H)}\big((-1)^{|E(H)|}k^{c(H)}+(-1)^{|E(H)|-1}k^{c(H-e_{C_1})}\big)=0.
\]
Namely, we arrive at $\sum_{H\in\mathcal{S}_1}(-1)^{|E(H)|}k^{c(H)}=0$ in \eqref{Eq2}.

Consider $\mathcal{S}_2$. Assume $B_2=C_2-e_{C_2}$ for some $\alpha$-compatible cycle $C_2$ of $G$, where $e_{C_2}$ is the maximal member in $C_2$ with respect to $\prec$. It follows from the ordering $e_{B_1}\preceq e_{B_2}$ that  $e_{B_1}\preceq e_{B_2}\prec e_{C_2}$. So, $B_1$ and $B_2$ do not contain $e_{C_2}$. Combining this, using the same argument as the case $\mathcal{S}_1$, we conclude that every element $H$ in $\mathcal{S}_2$ containing no $e_{C_2}$ gives rise to a unique member $H+e_{C_2}$ in $\mathcal{S}_2$ containing $e_{C_2}$, and vice versa. Hence, we also obtain $\sum_{H\in\mathcal{S}_2}(-1)^{|E(H)|}k^{c(H)}=0$ in \eqref{Eq2}.

Likewise, we can deduce the total contribution of the members in $\mathcal{S}_i$ to $P(G,\alpha;k)$ is zero for each $i=1,2,\ldots,q$. Therefore, $P(G,\alpha;k)$ in \eqref{Eq2} can be written as
\begin{equation}\label{Eq3}
P(G,\alpha;k)=\sum_{H\in\mathcal{S}_{q+1}}(-1)^{|E(H)|}k^{c(H)}.
\end{equation}
Since every element $H\in\mathcal{S}_{q+1}$ is an $\alpha$-compatible spanning subgraph of $G$ and contains no broken $\alpha$-compatible cycles, we can deduce that  $H$ contains no cycles. Therefore, $H$  is a forest. This implies that for any member $H\in\mathcal{S}_{q+1}$,
\[
|E(H)|\le r(G)\quad\And\quad c(H)=|V(G)|-|E(H)|.
\]
Thus, $P(G,\alpha;k)$ in \eqref{Eq3} can be further simplified to the following form
\[
P(G,\alpha;k)=\sum_{i=0}^{r(G)}(-1)^iw_i(G,\alpha)k^{|V(G)|-i},
\]
which completes the proof.
\end{proof}

As an immediate result of \autoref{Generalized-Broken-Circuit-Theorem}, we have the next corollary.
\begin{corollary}\label{Value-1}
Let $\alpha$ be an admissible assigning of $G$. Then $(-1)^{|V(G)|}P(G,\alpha;-1)$ is the number of $\alpha$-compatible spanning subgraphs of $G$ that do not contain  broken $\alpha$-compatible cycles with respect to a linear ordering $\prec$ on $E(G)$.
\end{corollary}

Notably, the cycle-assigning polynomial $P(G,\alpha;k)$ agrees with the chromatic polynomial $P(G;k)$ of $G$ whenever $G$ is $\alpha$-compatible. When $G$ is $\alpha$-compatible, the $\alpha$-compatible spanning subgraph, $\alpha$-compatible cycle and broken $\alpha$-compatible cycle of $G$ correspond to the ordinary spanning subgraph, cycle and broken cycle of $G$, respectively. Hence, \autoref{Generalized-Broken-Circuit-Theorem} deduces the following Whitney's Broken Cycle Theorem in \cite{Whitney1932} directly.
\begin{theorem}[\cite{Whitney1932}]\label{Broken-Cycle}
Write the chromatic polynomial as $P(G;k)=\sum_{i=0}^{r(G)}(-1)^iw_i(G)k^{|V(G)|-i}$. Then each unsigned coefficient $w_i(G)$ is the number of spanning subgraphs of $G$ that have exactly $i$ edges and do not contain broken cycles with respect to a linear ordering $\prec$ on $E(G)$.
\end{theorem}
\section{Coefficient comparisons and further explorations}\label{Sec4}
In this section, we primarily apply the Generalized Broken Cycle Theorem to explore how the coefficients of cycle-assigning polynomials vary with changes in admissible assignings and to handle certain group coloring problems in graphs. We first establish a unified order-preserving relation from admissible assignings to cycle-assigning polynomials when both are naturally ordered.  Let $\alpha$ be an assigning of $G$ and $\prec$ be a linear ordering on $E(G)$. For convenience, we denote by $\mathcal{C}(G,\alpha)$, $\mathcal{C}'(G,\alpha)$ and  $\mathcal{C}''(G,\alpha)$ the sets of $\alpha$-compatible cycles, broken $\alpha$-compatible cycles, and $\alpha$-compatible spanning subgraphs of $G$ that do not contain  broken $\alpha$-compatible cycles with respect to $\prec$, respectively. 
\begin{theorem}[Comparison of Coefficients]\label{Comparison}
Let $\alpha,\alpha'$ be two admissible assignings of $G$. If $\alpha(C)\le \alpha'(C)$ for all cycles $C\in \mathcal{C}(G)$, then the unsigned coefficients of the cycle-assigning polynomials $P(G,\alpha;k)$ and $P(G,\alpha';k)$ satisfy 
\[
w_i(G,\alpha)\le w_i(G,\alpha') \quad \mbox{ for }\quad i=0,1,\ldots,r(G).
\]
In particular, $w_i(G)=w_i(G,0)\le w_i(G,\alpha)$ for any admissible assigning $\alpha$ of $G$.
\end{theorem}
\begin{proof}
By \autoref{Generalized-Broken-Circuit-Theorem}, the proof  reduces to showing that every member $H$ of $\mathcal{C}''(G,\alpha)$ belongs to $\mathcal{C}''(G,\alpha')$. Given an element $H\in\mathcal{C}''(G,\alpha)$. As $H$ is an $\alpha$-compatible spanning subgraph of $G$ containing no  broken $\alpha$-compatible cycles, this implies that $H$ does not contain cycles. Furthermore, since $\alpha(C)\le\alpha'(C)$ for any cycle $C$ of $G$, we have $\mathcal{C}(G,\alpha)\supseteq\mathcal{C}(G,\alpha')$. This means $\mathcal{C}'(G,\alpha)\supseteq\mathcal{C}'(G,\alpha')$. This further deduces that $H$ does not contain  broken $\alpha'$-compatible cycles. Therefore, we conclude that $H$ is an $\alpha'$-compatible spanning subgraph of $G$ containing no  broken $\alpha'$-compatible cycles. Namely, $\mathcal{C}''(G,\alpha)\subseteq \mathcal{C}''(G,\alpha')$.
\end{proof}

When $G$ is a loopless graph, each unsigned coefficient $w_i(G)$ of $P(G;k)$ is a positive integer, which is easily obtained from \autoref{Broken-Cycle}. Applying \autoref{DCF} and \autoref{Comparison} to this, we derive that for any admissible assigning $\alpha$,  the coefficients of the cycle-assigning polynomial $P(G,\alpha; k)$ are nonzero and alternate in sign, unless $\alpha(e)=0$ for some loop $e$.
\begin{corollary}\label{Coefficient}
Let $\alpha$ be an admissible assigining of $G$. Then each unsigned coefficient $w_i(G,\alpha)$ of the cycle-assigning polynomial $P(G,\alpha;k)$ is a positive integer, unless $\alpha(e)=0$ for some loop $e$ of $G$. Moreover, $w_0(G,\alpha)=1$ whenever $P(G,\alpha;k)$ is not the zero polynomial.
\end{corollary}

According to the relationship $P(G,\alpha;k)=k^{c(G)}\tau(G,\alpha;k)$, it is worth noting that for any admissible assigning $\alpha$, the coefficients of the $\alpha$-assigning polynomial $\tau(G,\alpha;k)$ share the same properties as those of the cycle-assigning polynomial $P(G,\alpha;k)$, as outlined in \autoref{Decomposition}, \autoref{Generalized-Broken-Circuit-Theorem}, \autoref{Value-1}, \autoref{Comparison} and \autoref{Coefficient}.

It is well-known that the unimodality of the sequence $w_i(G)$ was conjectured by Read \cite{Read1968}, and the log-concavity was conjectured by Hoggar \cite{Hoggar1974}. Both of these conjectures have been confirmed by Huh et al. \cite{AHE2018,Huh2012,Huh-Katz2012}. Inspired by their work, we conjecture that the unsigned coefficients of the cycle-assigning polynomial $P(G,\alpha;k)$ form a log-concave, and hence unimodal, sequence for all graphs $G$ and almost all admissible assignings of $G$.
\begin{conjecture}\label{Conj}
Let $\alpha$ be an admissible assigning of $G$ such that $\alpha(e)=1$ for all loops $e$ of $G$. Then the sequence $w_0(G,\alpha),w_1(G,\alpha),\ldots,w_{r(G)}(G,\alpha)$ satisfies the following properties:
\begin{itemize}
\item The sequence $w_i(G,\alpha)$ is unimodal, i.e., there is an index $0\le j\le r(G)$ such that 
\[
w_0(G,\alpha)\le\cdots\le w_{j-1}(G,\alpha)\le w_j(G,\alpha)\ge w_{j+1}(G,\alpha)\ge\cdots\ge w_{r(G)}(G,\alpha).
\]
\item The sequence $w_i(G,\alpha)$ is log-concave, i.e, for $j=1,\ldots, r(G)-1$,
 \[
w_{j-1}(G,\alpha)w_{j+1}(G,\alpha)\le  w_j(G,\alpha)^2.
 \]
\end{itemize}
\end{conjecture}

It is a pleasant surprise that the Generalized Broken Cycle Theorem provides a new perspective on addressing certain group coloring problems in graphs. A well-known property in \cite{Tutte1949} says that if a graph has a nowhere-zero $A$-flow, then it admits a nowhere-zero $A'$-flow for any Abelian group $A'$ whose order is at least the order of $A$. This does not extend to the more general concept of group connectivity observed by Jaeger et al. \cite{Jaeger1992}. For a connected planar graph $G$, Jaeger et al. \cite{Jaeger1992} also showed that $G$ is $A$-colorable if and only if the dual graph $G^*$ is $A$-connected. Naturally, there is a corresponding problem in group colorings: for any Abelian groups $A_1$ and $A_2$ with the same order, is a $A_1$-colorable graph also $A_2$-colorable? 

Most recently, Langhede and Thomassen in \cite{Langhede-Thomassen2020} confirmed that the property holds if $A'$ is sufficiently large compared to $A$. This further  implied that $G$ is always $A$-colorable whenever $A$ has a sufficiently large order, which can be explained by the Generalized Broken Cycle Theorem. More specifically, note the basic fact that for a loopless graph $G$, $G$ is $A$-colorable for some Abelian group $A$ if and only if $P\big(D(G),f;A\big)>0$ for any $f:E(G)\to A$. This is also equivalent to saying that $G$ is $A$-colorable if and only if $P(G,\alpha_{D,f};|A|)>0$ for any $D$-admissible assigning $\alpha_{D,f}$ related to $f:E(G)\to A$. When $G$ is a loopless graph, it follows from \autoref{Generalized-Broken-Circuit-Theorem} that $w_0(G,\alpha)=1$ for any admissible assigning $\alpha$ of $G$. Thus, $P(G,\alpha;|A|)>0$ holds for all admissible assignings $\alpha$ of $G$ in this case whenever the Abelian group $A$ has a large enough order. 

We conclude this section with the next corollary, summarizing the above findings.
\begin{corollary}\label{Group-Connectivity}
Let $A,A'$ be two Abelian groups of the same order, $G$ be a $A'$-colorable graph and $D,D'$ be two orientations. If  each $D$-admissible assigning $\alpha_{D,f}$ of $G$ with $f:E(G)\to A$ is a $D'$-admissible assigning $\alpha_{D',f'}$ of $G$ for some $f':E(G)\to A'$, then $G$ is $A$-colorable.
\end{corollary}
\begin{proof}
Since $G$ is $A'$-colorable, we have $P\big(D'(G),f';A'\big)>0$ for any $f':E(G)\to A'$. As each $D$-admissible assigning $\alpha_{D,f}$ with $f:E(G)\to A$ is a $D'$-admissible assigning $\alpha_{D',f'}$ of $G$ for some $f':E(G)\to A'$,  it follows that for any $f:E(G)\to A$, 
\[
P\big(D(G),f;A\big)=P(G,\alpha;|A|)=P\big(D'(G),f';A'\big)>0
\] 
via \autoref{A=B}, where $\alpha=\alpha_{D,f}=\alpha_{D',f'}$. This implies that $G$ is $A$-colorable.
\end{proof}
\section*{Acknowledgements}
We are truly grateful to the anonymous referees for their insightful comments and valuable suggestions, which significantly improved this paper. \autoref{f-Tension-Polynomial} and \autoref{Other-Expression} are attributed to the referees' report and Kochol's work \cite{Kochol2024}. The work is supported by National Natural Science Foundation of China (Grant No. 12301424).

\end{document}